\documentclass[letter]{amsart}

\usepackage{amsmath, amsthm,  amsfonts, amssymb}
\usepackage{graphicx}
\usepackage{mathrsfs} 
\usepackage{array}
\usepackage{hyperref}
\usepackage[all]{xy}
\usepackage{mathtools}
\input xy 

\xyoption{all}
\xyoption{2cell} 
\xyoption{v2}

\DeclareMathOperator{\Aut}{Aut}

\DeclareMathOperator{\Map}{Map}
\DeclareMathOperator{\Ad}{Ad}

\DeclareMathOperator{\Eq}{Eq}

\DeclareMathOperator{\ev}{ev}
\DeclareMathOperator{\pr}{pr}
\DeclareMathOperator{\Det}{Det}

\theoremstyle{plain}
\newtheorem{theorem}{Theorem}[section]

\newtheorem{proposition}[theorem]{Proposition}

\theoremstyle{definition}
\newtheorem{definition}[theorem]{Definition}

\theoremstyle{remark}
\newtheorem{example}{Example}[section]

\newcommand{\cC}{{\mathcal C}}
\renewcommand{\cH}{{\mathcal H}}
\newcommand{\cG}{\mathcal{G}}

\newcommand{\cA}{{\mathcal A}}

\newcommand{\cF}{{\mathcal F}}

\renewcommand{\cL}{{\mathcal L}}
\newcommand{\cP}{{\mathcal P}}

\newcommand{\BB}{{\mathbb B}}
\newcommand{\RR}{{\mathbb R}}
\newcommand{\ZZ}{{\mathbb Z}}

\newcommand{\PP}{{\mathbb P}}
\renewcommand{\AA}{{\mathbb A}}
\renewcommand{\a}{\alpha}
\renewcommand{\b}{\beta}
\renewcommand{\c}{\gamma}
\renewcommand{\d}{\delta}
\newcommand{\vps}{\vphantom{*}}

\newcommand{\DD}{D \mspace{-1.5mu}D}

\begin{document}

\title[The FMS anomaly and  lifting bundle gerbes]{The Faddeev--Mickelsson--Shatashvili anomaly \\ and lifting bundle gerbes}

 \author[P. Hekmati]{Pedram Hekmati}
  \address[P. Hekmati]
  {School of Mathematical Sciences\\
  University of Adelaide\\
  Adelaide, SA 5005 \\
  Australia}
  \email{pedram.hekmati@adelaide.edu.au}

 \author[M. K. Murray]{Michael K. Murray}
  \address[M. K. Murray]
  {School of Mathematical Sciences\\
  University of Adelaide\\
  Adelaide, SA 5005 \\
  Australia}
  \email{michael.murray@adelaide.edu.au}
  
  \author[D. Stevenson]{Danny Stevenson}
\address[D. Stevenson]
{School of Mathematics and Statistics\\
University of Glasgow  \\
15 University Gardens\\
Glasgow G12 8QW\\
United Kingdom}
\email{danny.stevenson@glasgow.ac.uk}

 \author[R. F. Vozzo]{Raymond F. Vozzo}
  \address[R. F. Vozzo]
  {School of Mathematical Sciences\\
  University of Adelaide\\
  Adelaide, SA 5005 \\
  Australia}
  \email{raymond.vozzo@adelaide.edu.au}

\date{\today}

\thanks{The authors acknowledge the support of the Australian Research Council's Discovery 
Project Scheme and the Engineering and Physical Sciences Research Council. PH was partially supported under DP0878184, MM under DP0769986 and RV  under DP1092682. DS was supported by the Engineering and Physical Sciences Research Council [grant number EP/I010610/1]. We thank David Roberts for useful discussions and the name \lq structure group bundle' and 
Professor Nikita Nekrasov for clarifying remarks on  the history of the topic.}


%
% 	53C08  	Gerbes, differential characters: differential geometric aspects
%   	81T13  	Yang-Mills and other gauge theories
%   57R22  	Topology of vector bundles and fiber bundles
%


\begin{abstract}
In gauge theory, the Faddeev--Mickelsson--Shatashvili anomaly ari\-ses as a prolongation problem for the action of the gauge group on a bundle of projective Fock spaces. In this paper, we study this anomaly from the point of view of bundle gerbes and give several equivalent descriptions of the obstruction. These include lifting bundle gerbes with non-trivial structure group bundle and bundle gerbes related to the caloron correspondence.
\end{abstract}

\maketitle

\tableofcontents

\section{Introduction}
\label{sec:intro}

The Faddeev--Mickelsson--Shatashvili (FMS) anomaly arises in Hamiltonian quantisation of massless chiral 
fermions interacting with external gauge potentials. It signals the breakdown of local gauge 
symmetry in the quantum theory, which is required for identifying gauge 
equivalent fermionic Fock spaces and thereby removing unphysical degrees of freedom. 

The anomaly manifests itself in a variety of ways. Historically it first 
appeared as an anomalous term in the equal-time commutators of Gauss law generators \cite{FS84, Fad84, Mic85}. 
Globally this is due to the fact that the gauge group $\cG$ acts only projectively on the bundle 
of  Fock spaces parametrized by the space $\cA$ of gauge connections. It lifts to an honest action 
of an extension of the gauge group by the abelian group of circle-valued functions on $\cA$. 

In more detail, following the mathematical description given by Segal in  \cite{Seg},
we consider the chiral Dirac operator $D_{\! A}$ on a compact odd-dimensional Riemannian spin manifold,
coupled to a connection  $A \in \cA$.
This is an operator with discrete spectrum and a dense domain inside the Hilbert space $H$ of spinors.
Let $\cA_0 \subset \cA \times \RR$ be the subspace of all
pairs $(A, s)$ where $s$ is not in the spectrum of the Dirac operator. For every such pair $(A, s)$ the Hilbert space of spinors decomposes into the direct sum of the subspace $H^+_{(A, s)}$ spanned by the eigenspaces of $D_{\! A}$ for eigenvalues greater than $s$, and its orthogonal complement $H^-_{(A, s)}$. This splitting determines the vacuum in the fermionic Fock space
$$
F_{(A, s)} = \bigwedge H^+_{(A, s)} \otimes \bigwedge (H^-_{(A, s)})^*.
$$
Ideally one would like the Fock spaces to patch together to form a Hilbert bundle over $\cA$, but there is a phase ambiguity related to different choices of the parameter $s$. Indeed, if we leave $A$ fixed and consider another   $t>s$ not in 
the spectrum of $D_{\! A}$, then
$$
H = H^-_{(A, s)}\oplus V_{(A,s,t)} \oplus  H^+_{(A, t)}
$$
where $V_{(A,s,t)}$ is the finite dimensional vector space spanned by the eigenspaces for eigenvalues
between $s$ and $t$. This corresponds to shifting the vacuum level from $s$ to $t$. Moreover 
$$
H^+_{(A, s)} =  V_{(A,s,t)} \oplus H^+_{(A, t)} \quad\mbox{and}\quad
H^-_{(A, t)}  = H^-_{(A, s)} \oplus V_{(A,s,t)},
$$
and since $\bigwedge V^*_{(A,t,s)}\otimes\det V_{(A,t,s)} $
is canonically isomorphic to $\bigwedge V_{(A,t,s)}$, it follows that the Fock spaces are isomorphic
up to a phase
\begin{equation}\label{E:F=FxdetV}
F_{(A,s)} \simeq F_{(A,t)}\otimes \det V_{(A, s,t)}.
\end{equation}
The projectivizations $\PP(F_{(A,s)})$ and $\PP(F_{(A,t)})$ on the other hand can be identified for all $s, t \in \RR$ and descend to a projective bundle $\cP $ on $\cA$. We note that if the spectrum of the Dirac operator had a mass gap $(-m,m)$, it would be possible to fix a global vacuum level at $s=0$ for all connections. In the massless case however, the spectral flow of $D_{\! A}$ makes the vacuum section $A \to F_{(A, s)} $ discontinuous for any fixed $s$. The discontinuities occur exactly at those points $A\in\cA$ where an eigenvalue of $D_{\! A}$ crosses $s$. This makes it impossible to set a vacuum level once and for all, and one must instead resort to the local description above which gives rise to a projective bundle.

The FMS anomaly is tied to the question of whether or not there is a Hilbert bundle $\cH$ over the moduli space 
$\cA/\cG$ whose projective bundle is isomorphic to $\cP/\cG$. This question can be phrased in two equivalent ways. 
Firstly we note that $\cP \to \cA$ is always the projective bundle of a Hilbert bundle $\cH$ over $\cA$
because $\cA$ is contractible.  However to make $\cH$ a bundle on $\cA/\cG$ we need to lift the group action of $\cG$ to $\cH$ and the obstruction to that is a (locally smooth) group 2-cocycle with values in ${\Map}(\cA,U(1))$.
Equivalently the problem can be tackled directly on $\cA/\cG$ where it is well-known that the obstruction to a projective
bundle over $\cA/\cG$ being the projectivisation of a Hilbert bundle is a class in $H^3(\cA/\cG, \ZZ)$. The image of this class in real cohomology was first computed in \cite{CarMicMur97} using the Atiyah--Patodi--Singer index theorem. It was further shown in \cite{CarMur94}
that these two approaches are related by transgression. Namely the transgression of the three class in question yields a Lie algebra two-cocycle which is the derivative of the group cocycle.

Central to the discussions in \cite{CarMicMur97} was a line bundle $\Det \to  \cA_0$ and an associated short 
exact sequence formed from the automorphisms of $\Det$,
\begin{equation}
\label{eq:intro-exact}
1 \to \Map(\cA, U(1)) \to \widehat\cG \to \cG \to 1.
\end{equation}
  The primary purpose of this paper is to explain the 
observation (Proposition \ref{prop:motivate}) that  the FMS class vanishes if and  only if the $\cG$-bundle 
$\cA \to \cA/\cG$ lifts to a $\widehat \cG$-bundle.   Firstly, although \eqref{eq:intro-exact} is not a central extension of groups, we can 
apply the recently developed theory of the second author \cite{Mur} to characterise the obstruction 
to lifting $\cA \to \cA/\cG$ as a bundle gerbe on $\cA/ \cG$ with non-trivial structure group bundle.  The structure group bundle in question is a bundle
of abelian groups $\AA = \cA \times_\cG \Map(\cA, U(1))$ and the obstruction to the lift is the Dixmier--Douady class
of the bundle gerbe which is an element of 
the cohomology group
$
H^2(\cA/\cG,  \AA).
$
Secondly we relate this lifting bundle gerbe to the FMS gerbe from \cite{CarMur96} and show that 
their Dixmier--Douady invariants are related by isomorphisms
\begin{equation}
\label{eq:intro-iso}
H^2(\cA/\cG,  \AA) \simeq H^2(\cA/\cG,  U(1)) \simeq H^3(\cA/ \cG, \ZZ),
\end{equation}
which explains Proposition \ref{prop:motivate}. We also show how earlier results of the authors on the caloron correspondence \cite{HekMurVoz} can be used to construct a so-called caloron bundle gerbe, associated to the extension \eqref{eq:intro-exact}, which is stably isomorphic to the FMS gerbe.  

The paper is organised as follows. In Section \ref{S:FM gerbe} we review the theory of $U(1)$-bundle gerbes and the construction of the FMS gerbe from \cite{CarMur96}. We then establish our motivating result, Proposition \ref{prop:motivate} referred to above. In Section \ref{sec:lifting} we review the theory of bundle gerbes with non-trivial structure group bundle from \cite{Mur}, in particular we develop the theory of the lifting bundle gerbe of an abelian extension. In Section \ref{S:ev}  we establish the relationship between the lifting bundle gerbe with non-trivial structure group bundle and the FMS bundle gerbe. Finally in Section \ref{S:caloron gerbe} we discuss the caloron bundle gerbe and remark on open problems for which this point of view may be helpful.

\section{The Faddeev--Mickelsson--Shatashvili bundle gerbe}\label{S:FM gerbe}

\subsection{Bundle gerbes}
\label{sec:bg}

	We give a brief introduction to bundle gerbes here and refer the interested reader to \cite{Mur10} for an introduction and  \cite{Mur96Bundle-gerbes, MurSte00Bundle-gerbes:-stable} for further details.

	Let  $\pi \colon Y \to M$ be a surjective submersion.\footnote{Note that in \cite{Mur96Bundle-gerbes} $Y$ was required to be a fibration, however this is not actually necessary and we shall need the more general setup later.}
	  Let $Y^{[p]} \subset Y^p$ denote the fibre product, that is $(y_1, \dots, y_p) \in Y^{[p]}$ if and only if $\pi(y_1) = \pi(y_2) = \cdots =\pi(y_p)$. 
In the context of bundle gerbes it will be useful to call $Y$ the {\em object space} of the bundle gerbe.  For each $i = 1, \dots, p$
we  have the projection $\pi_i \colon Y^{[p]} \to Y^{[p-1]}$ which omits the $i^{\text{th}}$ element in the $p$-tuple.  In particular $\pi_1(y_1, y_2) = y_2$ and $\pi_2(y_1, y_2) = y_1$. 
If  $U \subset M$ is an open subset it will be useful to introduce the 
notation $Y_U  = \pi^{-1}(U) \subset Y$ for  the restriction of $Y$ to $U$.

If $Q$ and $R$ are two $U(1)$-bundles we define their \emph{product} $Q \otimes R$ to be  the quotient
of the fibre product of $Q$ and $R$ by the $U(1)$ action $(q, r)z = (qz, rz^{-1})$, with the induced right action 
of $U(1)$ on equivalence classes being given by $[q, r]w= [q, rw] = [qw, r]$. 
In addition if $P$ is a $U(1)$-bundle  we denote by $P^*$ the $U(1)$-bundle with the same total space as $P$ but with the action of $U(1)$ changed to its inverse, thus if 
$u\in P^*$ and $z\in U(1)$ then $z$ acts on $u$ by sending it to $uz^{-1}$.  We call $P^*$ the {\em dual} of $P$.

If $Q \to Y^{[p]}$ is a $U(1)$-bundle we define a new $U(1)$-bundle $\delta(Q) \to Y^{[p+1]}$
by
$$
\delta(Q) = \pi_1^*(Q) \otimes \pi_2^*(Q)^* \otimes \pi_3^*(Q) \otimes \cdots.
$$
It is straightforward to check that $\delta(\delta(Q)) $ is canonically trivial as a $U(1)$-bundle.

We then have the following definition.

\begin{definition}
A \emph{bundle gerbe} over $M$ is a pair $(P, Y)$ where $Y \to M$ 
is a surjective submersion and $P \to Y^{[2]}$ is a $U(1)$-bundle   satisfying the following two conditions:
\begin{enumerate}
\item There is a \emph{bundle gerbe multiplication} which is a smooth isomorphism
$$
m \colon \pi_3^*(P) \otimes \pi_1^*(P) \to \pi_2^*(P)
$$
of $U(1)$-bundles over $Y^{[3]}$.    
  \item This multiplication is associative, namely if   $P_{(y_1, y_2)}$ denotes the fibre of $P$ over $(y_1, y_2)$ then    the 
following diagram commutes for all $(y_1, y_2, y_3, y_4) \in Y^{[4]}$.

$$
\xymatrix{
P_{(y_1, y_2) } \otimes P_{(y_2, y_3) } \otimes P_{(y_3, y_4) }  \ar[r] \ar[d] & P_{(y_1, y_3) } \otimes P_{(y_3, y_4) }  \ar[d]\\
P_{(y_1, y_2) } \otimes P_{(y_2, y_4) } \ar[r] &  P_{(y_1, y_4) }\\
}
$$

\end{enumerate}
\end{definition}

We will find it convenient to depict a bundle gerbe with a diagram of the form:

$$
\xymatrix{
P \ar[d] &\\
Y^{[2]} \ar@<1ex>[r] \ar@<-1ex>[r] & Y \ar[d]\\
	&M
}
$$

Two bundle gerbes $(P, Y)$ and $(Q, Y)$ are called {\em isomorphic} if there is a $U(1)$-bundle isomorphism $P \simeq Q$
commuting with the bundle gerbe multiplication.

\begin{example}
If $Q \to Y$ is a $U(1)$-bundle then we can form $\delta(Q) \to Y^{[2]}$ as above. This has a natural bundle gerbe  multiplication given by 
the contraction
$$ 
c \colon Q_{y_2}  \otimes Q_{y_1}^* \otimes   Q_{y_3} \otimes Q_{y_2}^* \to  Q_{y_3}  \otimes Q_{y_1}^*.
$$ 

\end{example}

More generally a bundle gerbe $(P,Y)$ over $M$ is said to be \emph{trivial} if  there is a $U(1)$-bundle
$Q \to Y$ such that $(P,Y)$ is 
isomorphic to $ (\delta(Q),Y)$.  We call $Q$ and the isomorphism $\delta(Q) \simeq P$ a \emph{trivialisation} of $P$. Any two trivialisations of $P$ are related by tensoring with the pullback of a $U(1)$-bundle on $M$.

If $(P, Y)$ is a bundle gerbe over $M$ and $f \colon N \to M$, we can pullback $Y$ to a surjective submersion $f^*(Y) \to N$. 
We have an induced map $\hat f \colon (f^*(Y))^{[2]} = f^*(Y^{[2]}) \to Y^{[2]}$. The bundle gerbe product pulls 
back to $f^*(P) = \hat f^*(P) $ to define the pullback bundle gerbe $(f^*(P) , f^*(Y))$. 

If $(P,Y)$ is a bundle gerbe over $M$ then we can form the \emph{dual} bundle gerbe $(P^*,Y)$ by setting $P^*\to Y^{[2]}$ to be the dual of the $U(1)$-bundle $P$ in the sense 
described earlier.  The process of forming duals commutes with taking pullbacks and forming tensor products, so the bundle gerbe multiplication on $P$ 
induces a bundle gerbe multiplication on $P^*$ in a canonical way.  

Given two bundle gerbes $(P,Y)$ and $(Q,X)$ over $M$, we can form a new bundle gerbe $(P\otimes Q,Y\times_M X)$ over the same base called 
the \emph{tensor product} of $P$ and $Q$.  Here the surjective submersion 
is  the fibre product $Y\times_M X\to M$ and $P\otimes Q$ is the $U(1)$-bundle on $(Y\times_M X)^{[2]}$ whose fibre at $((y_1,x_1),(y_2,x_2))$ is given 
by 
$$ 
P_{(y_1,y_2)}\otimes Q_{(x_1,x_2)}.   
$$
The bundle gerbe multiplication on $P\otimes Q$ is defined in the obvious way, using the bundle gerbe multiplications on $P$ and $Q$.  

Every bundle gerbe $(P,Y)$ over $M$ has a characteristic class  $\DD(P,Y) \in H^3(M,{\mathbb Z})$   called the \emph{Dixmier--Douady} class.      We construct it in terms of \v{C}ech cohomology as follows.     Choose a good cover \cite{BotTu} ${\mathcal U} = \{U_\alpha\}$ of $M$   with 
sections $s_\alpha \colon U_\alpha \to Y$ of $\pi\colon Y\to M$.      Then 
$$
(s_\alpha, s_\beta) \colon U_\alpha \cap U_\beta \to Y^{[2]}
$$
is a section.      Choose sections $\sigma_{\alpha\beta} $ of $P_{\alpha\beta} = (s_\alpha, s_\beta)^*(P)$. These are maps
$$
\sigma_{\alpha\beta} \colon U_\alpha \cap U_\beta \to P
$$
with  $\sigma_{\alpha\beta}(x) \in P_{(s_\alpha(x), s_\beta(x))}$.      Over triple overlaps we have 
$$
m(\sigma_{\alpha\beta}(x), \sigma_{\beta\gamma}(x)) =  \sigma_{\alpha\gamma}(x) g_{\alpha\beta\gamma}(x)\in P_{(s_\alpha(x), s_\gamma(x))}
$$
for  $g_{\alpha\beta\gamma} \colon U_\alpha \cap U_\beta \cap U_\gamma \to U(1)$.       This defines a 
cocycle  which represents the Dixmier--Douady class
$$
\DD(P, Y) = [g_{\alpha\beta\gamma}] \in H^2(M, U(1)) \simeq H^3(M, {\mathbb Z}).
$$
Here $H^2(M,U(1))$ denotes the \v{C}ech cohomology of $M$ with coefficients in the sheaf 
of germs of maps from $M$ into $U(1)$.  
The Dixmier--Douady class of $P$ is the obstruction to $(P,Y)$ being trivial. That is,  $\DD(P,Y) = 0$  if and only if $(P,Y)$ is isomorphic to a trivial 
bundle gerbe.  Note also that the Dixmier--Douady class satisfies  $\DD(P\otimes Q,Y\times_M X) = 
\DD(P,Y) + \DD(Q,X)$ and $\DD(P, Y) = - \DD(P^*, Y)$. 

Two bundle gerbes $(P, Y)$ and $(Q, X)$ over $M$ are said to be {\em stably isomorphic} if the bundle gerbe $(P \otimes Q^*, Y\times_M X)$ is trivial or, equivalently if $\DD(P, Y) = \DD(Q, X)$ \cite{MurSte00Bundle-gerbes:-stable}. 

We note here a standard result that will be needed later.

\begin{proposition}[\cite{MurSte00Bundle-gerbes:-stable}]\label{P:stable isos}
Let  $X \to M$ and $\pi \colon Y \to M$ be two surjective submersions and $\mu \colon X \to Y$ a map commuting 
with the projections to $M$.  Denote by $\mu \colon X^{[2]} \to Y^{[2]}$ the induced map. If $(P, Y)$ is 
a bundle gerbe on $M$, then $(\mu^*(P), X)$ is also a bundle gerbe which is stably 
isomorphic to $(P, Y)$.
\end{proposition}

Finally notice that everything we have said here generalises if $U(1)$ is replaced by any abelian topological group $A$. The only 
modification required is that the Dixmier--Douady class is in the \v{C}ech cohomology group $H^2(M, A)$ rather than $H^2 (M, U(1))$.

\subsection{The Faddeev--Mickelsson--Shatashvili bundle gerbe}\label{SS:FM gerbe}

The space $\cA_0$ introduced in Section \ref{sec:intro}
projects onto  the space $\cA$ of connections. This projection
is clearly onto and a submersion because locally $\cA_0$ is a
product of an open set in $\cA$ and an open set in $\RR$.
Note though that  $\cA_0 \to \cA$ is not a fibration
because it fails to be  locally trivial near points of $\cA$
for which $D_{\! A}$ has degenerate eigenvalues.

The fibre product $\cA_0^{[2]}$ can be identified with the set of 
triples $(A, s, t)$ where neither of the real
numbers $s$ and $t$ are in the spectrum of $D_{\! A}$.
We define a line bundle $\cF$ over $\cA_0^{[2]}$  pointwise by defining
its fibre to be 
\begin{equation*}
	\cF_{(A,  s, t)} = \begin{cases} \det V_{(A,s, t)} & \text{if} \ s \leq t \\ 
	\det V_{(A,s,t)}^* & \text{if} \  s \geq t ,\end{cases}
\end{equation*}
where as before $V_{(A, s, t)}$ is the sum of all the eigenspaces for eigenvalues between $s$ and $t$. The first two
conditions for a bundle gerbe follow naturally. For the
third consider $r < s < t$ then
$$
V_{(A, r, s)} \oplus V_{(A,s,  t)} = V_{ (A, r, t)}
$$
so that
$$
\det(V_{(A, r, s)}) \otimes \det( V_{(A, s, t)}) =  \det(V_{(A, r, t)})
$$
giving us the bundle gerbe multiplication
$$
\cF_{(A, r, s)} \otimes  \cF_{(A, s, t)} \simeq  \cF_{(A, r, t)}.
$$
We call the bundle gerbe $(\cF, \cA_0)$ the \emph{trivial Faddeev--Mickelsson--Shatashvili (FMS) bundle
gerbe on $\cA$. }

We know that $(\cF, \cA_0)$ is trivial because $\cA$ is contractible. In \cite{CarMur96}
an explicit trivialisation was constructed using a determinant line bundle 
over an infinite-dimensional Grassmannian.  The details of that construction are not important 
in what follows.  We denote this trivialisation by $\Det \to \cA_0$ 
and note that we must have 
\begin{equation}
\label{eq:FMtriv}
\cF^{\vps}_{(A, s, t)} \simeq \Det_{(A, s)}^* \otimes \Det^{\vps}_{(A, t)}.
\end{equation}
From  equation (\ref{E:F=FxdetV}) we have
$$
F_{(A, s)} \simeq F_{(A, t)} \otimes \cF_{(A, s, t)}
$$
and hence 
$$
F_{(A, s)}\otimes \Det_{(A, s)}  \simeq F_{(A, t)} \otimes \Det_{(A, t)}.
$$
This means that the Hilbert bundle $F \otimes \Det \to \cA_0$ descends to a Hilbert bundle $\cH$ on $\cA$ and
clearly we have $\PP(\cH) \simeq \cP$.
Therefore if the bundle gerbe $\cF$ is trivial, then the
projective bundle $\cP$ is the
projectivisation of a Hilbert bundle. The converse is easily seen to be true \cite{CarMur94}.

So far everything we have said is happening on $\cA$ where
the  bundle gerbe is trivial.  However $\cG$
clearly acts on $\cF$ and it descends to a bundle
gerbe $(\cF/\cG, \cA_0 / \cG)$  on $\cA/ \cG$, which we call the \emph{Faddeev--Mickelsson--Shatashvili  bundle gerbe}, or simply the \emph{FMS bundle gerbe}.
A trivialisation of the bundle gerbe on $\cA/\cG$ is
therefore equivalent to a $\cG$-equivariant trivialisation
of the bundle gerbe over $\cA$.  As $\cA$ is contractible any two trivialisations of a bundle gerbe 
are isomorphic, so we conclude that the FMS bundle gerbe is trivial if and only if $\Det \to \cA_0$ 
admits an action of $\cG$ covering the action on $\cA_0$ and compatible with the isomorphisms $\cF \simeq \delta(\Det)$. 
 Denote by $\widehat\cG$ the 
group of all pairs $(\psi, g)$ where $\psi \in \Aut(\Det)$ is a right bundle automorphism covering the right action 
of $g \colon \cA_0 \to \cA_0$ and preserving the trivialisation $\delta(\Det) \simeq \cF$.  We have an exact sequence
\begin{equation}
\label{eq:exact}
1 \to \Map(\cA, U(1)) \to \widehat\cG \to \cG \to 1
\end{equation}
which admits local sections. We conclude that the FMS bundle gerbe is trivial if and only if this exact sequence splits. 

These observations enable us to prove the following Proposition which is the motivation for our constructions below. 

\begin{proposition}
\label{prop:motivate}
The principal $\cG$-bundle $\cA \to \cA/\cG$ lifts to a $\widehat\cG$-bundle if and only if $(\cF/\cG,  \cA_0/\cG)$ is trivial.
\end{proposition}

\begin{proof}
We have already seen that $(\cF/\cG,  \cA_0/\cG)$ is trivial if and only if \eqref{eq:exact} splits. 
Clearly if that occurs we have a homomorphism $\cG \to \widehat \cG$ which enables us to lift $\cA \to \cA / \cG$
to a $\widehat \cG$-bundle.

Conversely assume that $\cA$ can be lifted to a $\widehat\cG$-bundle $\widehat\cA \to \cA/\cG$. Then $\widehat\cA \to \cA$ 
is a $\Map(\cA, U(1))$-bundle which must be trivial.  Let $s \colon \cA \to \widehat\cA$ be a section and
let $A \in \cA$ and $g \in \cG$. We denote the action of $\cG$ on $\cA$ by $A \mapsto A^g$. Consider $s(A)$ and $s(A^g)$. As the following diagram commutes,
\begin{equation*}  
\xy 
(-15,7.5)*+{\widehat\cA}="1"; 
(10,0)*+{\cA}="2"; 
(-15,-7.5)*+{\cA/\cG}="3"; 
{\ar_p "1";"2"}; 
{\ar_s@/_1pc/ "2";"1"}; 
{\ar_{\widehat\pi} "1";"3"}; 
{\ar_\pi "2";"3"}; 
 \endxy
\end{equation*}
we have $$\widehat\pi(s(A^g)) = \pi (p (s(A^g))) = \pi(A^g) = \pi(A) = \pi (p (s(A))) = \widehat\pi(s(A))$$ and hence $s(A^g) = s(A)\phi(A, g)$ where
$\phi \colon \cA \times \cG \to \widehat\cG$.  It is easy to see that $\phi(A, g) \phi(A^g, h) = \phi(A, gh)$.
We can now define an action of $\cG$ on $\Det$.  If $g \in \cG$ and $l \in \Det_{(A, s)}$, the fibre of $\Det$ over $(A, s) \in \cA_0$,   we  define $l g = l \phi(A, g)$.  It is straightforward to check that this is an action and hence \eqref{eq:exact} splits.
\end{proof}

\section{Lifting bundle gerbes for abelian extensions}
\label{sec:lifting}

Proposition \ref{prop:motivate} motivates us to attempt to understand the obstruction to lifting $\cA \to \cA/\cG$ to 
a $\widehat \cG$-bundle. In the case of a {\em central} extension 
$$
1 \to Z \to \widehat H \to H \to 1 
$$
it is well known that the obstruction to lifting an $H$-bundle $Y \to M$ to an $\widehat H$-bundle is a class in $H^2(M, Z)$. This can 
be interpreted as the Dixmier--Douady class of the so-called lifting bundle gerbe \cite{Mur96Bundle-gerbes} defined as follows. 
As $Y $ is a principal $H$-bundle we have a map $\tau \colon Y^{[2]} \to H$ defined by $y_1 \tau(y_1, y_2) = y_2$ and we can 
use this to pullback the $Z$-bundle $\widehat H \to H$. The resulting bundle has a bundle gerbe product induced by 
the group action of $\widehat H$.  It is straightforward to check that a trivialisation of the lifting
bundle gerbe, which is a  $Z$-bundle $\widehat Y \to Y$, is precisely a lift of $Y$ to $\widehat H$. 

The problem with applying the theory of lifting bundle gerbes outlined above is that the exact sequence \eqref{eq:exact} is not a central extension but only an abelian extension, that 
is $\Map(\cA, U(1))$ is a normal, abelian subgroup of $\widehat \cG$ but not in the centre of $\widehat \cG$.  Recent work 
of the second author \cite{Mur}, which we now review,  has shown that for such extensions the obstruction 
to lifting a bundle is a bundle gerbe with non-trivial structure group bundle.

\subsection{Lifting bundle gerbes with non-trivial structure group bundle}
\label{SS:non-constant lifting}

\subsubsection{Bundle gerbes with non-trivial structure group bundle}

Let $\AA \to M$ be a locally trivial bundle of abelian groups over $M$.  We call such objects abelian group bundles. We say that a fibre bundle $P \to M$ is a principal $\AA$-bundle
if each fibre of $P \to M$ is a principal space for the corresponding fibre of $\AA \to M$ and if whenever we locally trivialise
$\AA$ as $\AA_U = U \times A$, we have that $P_U$ is a locally trivial principal $A$-bundle.  In such a case we call $\AA$ the structure group bundle of $P$. Duals and products of $\AA$-bundles are defined fibre by fibre. It is  straightforward to  show that isomorphism classes of $\AA$-bundles are classified by the group $H^1(M, \AA)$, where here $\AA$  also denotes the sheaf of smooth sections of the group bundle $\AA$.  

We can generalise the definition of bundle gerbes to the case of non-trivial structure group bundles as follows.  Let $Y \to M$ be a submersion and $\AA \to M$ be a bundle
of abelian groups with fibre isomorphic to $A$.  Denote also by $\AA$ the pullback of $\AA$ to any of the fibre products 
$Y^{[p]} \to M$. The definition of an $\AA$-bundle gerbe is then an $\AA$-bundle $Q \to Y^{[2]} $ with the obvious
notion of a bundle gerbe product. The definition of the Dixmier--Douady class is analogous to the $U(1)$-bundle gerbe case: choose a
good cover  ${\mathcal U} = \{U_\alpha\}$ of $M$ and sections $s_\a \colon U_\a \to Y$. Let $\sigma_{\a\b} \colon U_\a \cap U_\b \to Q$
be sections of $Q$. Then 
$$
\sigma_{\a\b} \sigma_{\b\c} =  \sigma_{\a\c} g_{\a\b\c}
$$
where $g_{\a\b\c} \colon U_\a \cap U_\b \cap U_\c \to \AA$. The class of $g_{\a\b\c}$ in $H^2(M, \AA)$
is the Dixmier--Douady class.

\subsubsection{Changing the structure group bundle of a  bundle gerbe.} 
\label{sec:change}
Let $\phi \colon \AA \to \BB$ be a homomorphism of group bundles, that is, $\phi$ is a bundle map which is a homomorphism 
on fibres $\phi_m \colon \AA_m \to \BB_m$ and moreover this homomorphism on fibres is, up to isomorphisms, constant. 
If $(P, Y)$ is an $\AA$-bundle gerbe over $M$, then we can extend the structure group bundle to $\BB$ by defining an associated bundle 
$P \times_\AA \BB$ where the action of $\AA$ on the left of $\BB$ is induced by $\phi$.  It is straightforward
to check that the bundle gerbe product extends. We denote this bundle gerbe by $(\phi_*(P), Y)$.  The homomorphism $\phi$ induces a homomorphism 
of sheaves of smooth sections of $\AA$ and $\BB$ which induces a map 
$$
\phi_* \colon H^2(M, \AA) \to H^2(M, \BB)
$$
and it is a straightforward calculation to show that 
\begin{equation}
\label{eq:DD-change}
\DD((\phi_*(P), Y)) = \phi_*( \DD (P, Y)).
\end{equation}

\subsubsection{Lifting bundle gerbes and abelian extensions}
Consider an extension of groups 
$$
1 \to A \to \widehat H \stackrel{\pi}{\to} H \to 1 
$$
where  $A$ is abelian and normal in $\widehat H$ but possibly not central.

 Let $Y \to M$ be a principal $H$-bundle and  $\tau \colon Y^{[2]} \to H$ as before. Let $L = \tau^*(\widehat H)$ so that an element 
of $L_{(y_1, y_2)} $ is a triple $(y_1, y_2, \hat h)$ where $\hat h \in A\widehat \tau(y_1, y_2)$ and
$\hat \tau(y_1, y_2)$ is any lift of $\tau(y_1, y_2) \in H$ to $\widehat H$.  For convenience
we identify 
$$
L_{(y_1, y_2)} = A\widehat \tau(y_1, y_2).
$$
Because $A$ is normal, the product on $\widehat H$ restricts to a well-defined map
\begin{equation}
\label{eq:productL}
L_{(y_1, y_2)} \times L_{(y_2, y_3)} \to L_{(y_1, y_3)}.
\end{equation}

Notice that $H$ acts on $A$ because we can lift $h \in H$ to $\hat h \in \widehat H$ and define $h(a) = \hat h a \hat h^{-1}$
which is independent of the choice of lift as $A$ is abelian. This action is non-trivial exactly when $A$ is not central. 
As a result we can form a group bundle $\AA = Y \times_H A\to M$.

 We define an action of (the pullback of) $\AA$ on $L$
 as follows.  Let $l \in L_{(y_1, y_2)}$ and $[y_2, a] \in \AA$, and  define $l [y_2, a] = la$. Notice
 that if $\pi(\hat h ) = h$, then
 $$
 l [y_2  h, a] = l [y_2, h(a)] = l [y_2, \hat h a \hat h^{-1} ] = l \hat h a \hat h^{-1}.
 $$
 This makes $L$ into an $\AA$-bundle. If $l \in L_{(y_1, y_2)}$ and $l' \in L_{(y_2, y_3)}$, then using 
 the product \eqref{eq:productL} we have $l l' \in  l_{(y_1, y_3)}$. 
Moreover 
 \begin{align*}
l [y_2, a] (l' [y_2, a]^{-1})  &=  l [y_2, a] l' ([y_2, a^{-1}])\\ 
                                &=  l [y_2, a] l' ([y_3, \tau(y_2, y_3)^{-1}(a^{-1})]) \\
                                 &=  la l' \tau(y_2, y_3)^{-1}(a^{-1}) \\
                                 & = l l'
\end{align*}
because  $ \tau(y_2, y_3)^{-1}(a^{-1}) = (l')^{-1} a^{-1} l'$. This shows
that the product descends to make $(L, Y)$ an $\AA$-bundle gerbe which we call the \emph{lifting bundle gerbe} of $Y \to M$.

Before we prove the next Proposition we need to collect some facts about right principal $A$-spaces for an abelian 
group $A$. Let  $X$ be such a space and $X^*$ the dual space.  There is a well-defined map 
$$
X^* \times X \to A
$$
which  we write as $(\xi, x) \mapsto \xi(x)$, where $\xi(x) $ is defined by  $x = \xi\, \xi(x)$ bearing 
in mind that $X^*$ is the same set as $X$ but with the inverse $A$-action.   We have 
$x a = \xi \, \xi(x) a $ so that $\xi(xa) = \xi(x) a$, and $x = \xi\, a^{-1}\,  a \,\xi(x) = (\xi a)\, a\, \xi(x)$
which implies that $(\xi a)(x) = \xi(x) a$. 

Let  $X_1, X_2, X_3$ be right principal $A$-spaces and 
define the map
$$
c \colon X_2^{\vps} \otimes X_1^* \times  X_3^{\vps} \otimes X_2^* \to X_3^{\vps} \otimes X_1^* 
$$
by
$$
c(  x_2 \otimes \xi_1 , x_3 \otimes \xi_2 ) = ( x_3 \otimes \xi_1 )  \xi_2(x_2) .
$$
Notice that this satisfies $$c(  (x_2 \otimes \xi_1)a , (x_3 \otimes \xi_2)a^{-1} )
= c(  x_2 a \otimes \xi_1 , x_3 \otimes \xi_2 a) = c(  x_2 \otimes \xi_1 , x_3 \otimes \xi_2 )$$
and therefore descends to a map
$$
c \colon X_2^{\vps} \otimes X_1^* \otimes  X_3^{\vps} \otimes X_2^* \to X_3^{\vps} \otimes X_1^* .
$$
Moreover if $\a = x \otimes \zeta  \in X \otimes Z^*$ and $z \in Z$, then define $\a(z) = x \zeta(z)$.
In particular if $\a \in X_2^{\vps} \otimes X_1^*$, $\b \in X_3^{\vps} \otimes X_2^*$ and $x \in X_1$, we have 
\begin{equation}
\label{eq:reversal}
c(\a, \b)(x) =   \b ( \a (x)).
\end{equation}

  \begin{proposition}
 The lifting bundle gerbe $(L, Y) $ is a trivial $\AA$-bundle gerbe if and only if the bundle $Y \to M$ lifts
 to $\widehat H$.
 \end{proposition}
\begin{proof} 
 First assume that the bundle
 lifts to $\widehat Y \to M$ with a map $\widehat Y \to Y $ which is an $A$-bundle.  We make $\widehat Y \to Y$
 into an $\AA$-bundle by defining  $\hat y [ y, a]  = \hat y a $, where  $\hat y \in \widehat Y$ and $\pi(\widehat y) = y$.
 
 We need to define an isomorphism $\phi \colon  L \to \delta(\widehat Y)$  of  $\AA$-bundles over $Y^{[2]}$ which preserves the bundle gerbe multiplications.
 If $\hat h \in L_{(y_1, y_2)} = A \tau(y_1, y_2)$ we choose $\hat y \in \widehat Y_{y_1}$ and notice
 that $\hat y \hat h \in \widehat Y_{y_1 \tau(y_1, y_2)} = \widehat Y_{y_2}$. Define the map $\phi$ by 
 $$
\phi(  \hat h )   =   \hat y  \hat h  \otimes \hat y^*  \in  \widehat Y_{y_2}^{\vps}  \otimes \widehat Y_{y_1}^*  .         
$$
This is well-defined because changing $\hat y$ to $\hat y [y_1, a]$ gives
\begin{align*}
  (\hat y[y_1, a])  \hat h  \otimes (\hat y[y_1, a])^* &=  (\hat y[y_1, a])  \hat h  \otimes \hat y^* [y_1, a^{-1}]  \\
                   &=   \hat y a  \hat h [y_1, a^{-1}] \otimes \hat y^*   \\ 
                    &=   \hat y a  \hat h [y_2, \hat h^{-1} a^{-1} \hat h] \otimes \hat y^* \\
                    & =  \hat y a \hat h \hat h^{-1} a^{-1} \hat h  \otimes \hat y^* \\\
                    & =  \hat y \hat h \otimes \hat y^*.
                      \end{align*}
To check that this an $\AA$-bundle isomorphism we note that 
$ \hat h [y_2, a] = \hat h a $ so that 
$$
\phi(\hat h [y_2, a]) =   \hat y (\hat h a) \otimes \hat y^*  = ( \hat y \hat h ) a
\otimes \hat y^* = \hat y \hat h [y_2, a] \otimes \hat y^* = (\hat y \hat h  \otimes \hat y^*) [y_2, a].
$$
To see that $\phi$ preserves multiplication choose $\hat k \in L_{(y_2, y_3)} = A \tau(y_2, y_3)$ and 
  $\hat y \hat h \in \widehat Y_{y_2}$ so that
\begin{align*}
c(\phi(\hat h),  \phi (\hat k))  &= c(\hat y \hat h \otimes \hat y^* , ( (\hat y \hat h) \hat k ) \otimes (\hat y \hat h)^*) \\
                            &= ((\hat y \hat h) \hat k ) \otimes \hat y^* \\
                             &= \hat y (\hat h  \hat k ) \otimes \hat y^* \\
                             & = \phi(\hat h \hat k)
    \end{align*}
as required.  It follows that $\widehat Y \to Y$ is a trivialisation of $L$.

On the other hand assume that the lifting bundle gerbe $L$ is trivial. Namely, there is an $\AA$-bundle $p \colon \widehat Y \to Y$
and an $\AA$-bundle isomorphism $\phi \colon  L \to \delta(\widehat Y)$ which commutes with the bundle gerbe products. 
Consider $\hat h \in \widehat H$ and $\hat y \in \widehat Y_y$ so that  $p(\hat y) = y$. We want to define an action of 
$\hat h$ on the right of $\hat y$ sending it to an element in $\widehat Y_{y h}$.  Notice that since
$\hat h \in L_{(y, y h)}$ and
$$
\phi \colon L_{(y, yh)}^{\vps} \to \widehat Y_{yh}^{\vps} \otimes \widehat Y_{y}^*,
$$
we can define 
$$
\hat y \hat h = \phi(\hat h)(\hat y).
$$
We need to check that this is a right group action. Let $\hat k   \in \widehat H$ and 
consider 
$$
\phi \colon L_{(yh, yhk)}^{\vps} \to \widehat Y_{yhk}^{\vps} \otimes \widehat Y_{yh}^*.
$$
Using \eqref{eq:reversal}, we have
\begin{align*}
(\hat y \hat h) \hat k &= \phi(\hat k) (\hat y \hat h) \\
                     &= \phi(\hat k) ( \phi (\hat h) (\hat y)) \\
                     & = c( \phi(\hat h) , \phi(\hat k)) (\hat y) \\
                     & = \phi(\hat h \hat k) (\hat y)\\
                    & = \hat y (\hat h \hat k).
              \end{align*}
Hence $\widehat Y \to Y$ is a lift of $Y$ to an $\widehat H$-bundle as required.
\end{proof}
In the  FMS case the structure group bundle is given by
$$
\AA = \cA \times_\cG \Map(\cA, U(1))
$$
and we denote the lifting bundle gerbe for $\cA \to \cA/ \cG$ by $(\cL, \cA)$.

\section{The lifting bundle gerbe and the Faddeev--Mickelsson--Shatashvili gerbe}\label{S:ev}

Recall that we were motivated by Proposition \ref{prop:motivate} to find a relationship between the 
FMS bundle gerbe and the problem of lifting the $\cG$-bundle $\cA \to \cA/\cG$ to $\widehat \cG$.
We have now described the obstruction to this lifting problem as a bundle gerbe over $\cA/\cG$. However it is not  a 
$U(1)$-bundle gerbe like the FMS  gerbe, but a bundle gerbe with structure group bundle $\AA = \cA \times_\cG \Map (\cA, U(1))$. 
So the answer cannot be that the lifting bundle gerbe is stably isomorphic to the FMS bundle gerbe. 
On the other hand if we consider the Dixmier--Douady invariant of the lifting bundle gerbe, we see that 
\begin{align*}
H^2(\cA / \cG , \cA \times_\cG \Map (\cA, U(1))) &\simeq H^3(\cA / \cG , \cA \times_\cG \Map (\cA, \ZZ))\\
&\simeq H^3(\cA / \cG , \cA \times_\cG \ZZ)\\
&\simeq H^3(\cA / \cG , \ZZ),
\end{align*}
where we have used the exact sequence
$$
1 \to \Map (\cA, \ZZ) \to \Map (\cA, \RR) \to \Map (\cA, U(1)) \to 1
$$
for the first isomorphism, and the fact that $\cA$ is connected (so that $\Map (\cA, \ZZ) = \ZZ$) for the second. 
Hence the obstruction class to the lifting problem lives in the same space as the FMS class. 
Below we will show that these classes are in fact equal. This requires changing the structure group bundle from $\cA \times_{\cG} \Map (\cA, U(1))$ to the trivial $U(1)$ group bundle  using the ideas in Section \ref{sec:change}.

Consider the evaluation map 
$$
\ev \colon \cA \times \Map(\cA, U(1)) \to U(1).
$$
 This descends to a homomorphism of group bundles
$$
\ev \colon \cA \times_\cG \Map(\cA, U(1)) \to \cA  \times_\cG U(1) = \cA/\cG \times U(1) ,
$$ sending $[A,f]$ to $[A, f(A)] = (\pi(A),f(A))$.
Similarly we can use the evaluation map on $\RR$-valued and $\ZZ$-valued maps to form a commuting diagram
$$
\xymatrix{
1  \ar@{=}[d]\ar[r] & \ar[r] \cA/\cG \times \ZZ  \ar[r] \ar@{=}[d]  & \cA \times_\cG \Map (\cA, \RR) \ar[r]\ar[d]_\ev & \cA \times_\cG \Map (\cA, U(1)) \ar[r]\ar[d]_\ev & 1 \ar@{=}[d] \\
1 \ar[r]& \ar[r] \cA/\cG \times \ZZ  \ar[r] & \cA/\cG \times \RR \ar[r] & \cA/\cG \times U(1)  \ar[r] & 1
}
$$
Taking the cohomology of this diagram shows that $\ev_*$ is an isomorphism
$$
H^2(M, {\cA\times_\cG\Map(\cA,U(1))}) \xrightarrow{\ev_*}  H^2(\cA/\cG, U(1)) \xrightarrow{\sim} H^3(\cA/\cG, \ZZ).
$$
It follows from equation \eqref{eq:DD-change} that
$$
 \DD( \ev_*(\cL), \cA) = \ev_*(\DD(\cL, \cA))
$$
and because $\ev_*$ is an isomorphism, we have that $ \DD( \ev_*(\cL), \cA) = 0$ if and only if $\DD(\cL, \cA) = 0$. 
It remains to show that the $U(1)$-bundle gerbe $( \ev_*(\cL), \cA)$ is stably isomorphic to the FMS bundle gerbe, so we can conclude that  $\DD( \ev_*(\cL), \cA) $ is equal to the FMS class and thereby explain
Proposition \ref{prop:motivate}.  We wish to prove then the following:

\begin{theorem}\label{T:ev(L) = FM*}
The bundle gerbe $( \ev_*(\cL), \cA)$ is stably isomorphic to the FMS bundle gerbe $(\cF/\cG, \cA_0/\cG)$.
\end{theorem}

\begin{proof}

Consider the following diagram:

\begin{equation*} 
\label{lifting} 
\xy 
(-20,7)*+{\cA}="1"; 
(0,10)*+{\cA_0}="2"; 
(20,7)*+{\cA_0/\cG}="3"; 
(0,-7)*+{\cA/\cG}="4";
(-20,21)*+{\cA^{[2]}}="5"; 
(0,24)*+{\cA_0^{[2]}}="6"; 
(20,21)*+{(\cA_0/\cG)^{[2]}}="7"; 
(-25, 34)*+{\ev_*(\cL)} = "8";
(25, 34)*+{\cF/\cG} = "9";
{\ar "1";"4"};
{\ar^\pi "2";"4"}; 
{\ar "3";"4"};
{\ar@<1ex> "5";"1"};
{\ar@<-1ex> "5";"1"};
{\ar@<1ex>^{\pi_2} "6";"2"};
{\ar@<-1ex>_{\pi_1} "6";"2"};
{\ar@<1ex> "7";"3"};
{\ar@<-1ex> "7";"3"};
{\ar_\a "2";"1"};
{\ar^\b "2";"3"};
{\ar "6";"5"};
{\ar "6";"7"};
{\ar "8" ; "5"};
{\ar "9" ; "7"}
\endxy
\end{equation*}
We claim that 
the bundle gerbe $(\a^*(\ev_*(\cL)), \cA_0)$ is isomorphic to the bundle gerbe 
$(\pi_1^*(\Det)^* \otimes \pi_2^*(\Det) \otimes \b^*(\cF/\cG) , \cA_0) = 
(\delta (\Det)^* \otimes \b^*(\cF/\cG) , \cA_0)$.

To see this, take $\hat g \in \widehat \cG$ which projects to $g \in \cG$ and let $A \in \cA$ and $z \in U(1)$. An element of $\ev_*(\cL)$ is  given by a triple $[(A, \hat g), z] = [(A, f\hat g), f(A)z]$ for any $f\in \Map(\cA,U(1))$.  Notice that an element of $\cA_0^{[2]}$ takes the form $((A, s), (A^g, t))$ and 
an element of the fibre of  $\a^* \ev_*(\cL)$ over such a point is a pair $([(A, \hat g), z], (A, s, t))$, where $(A, s)$ and $(A, t)$ belong to $\cA_0$. 

On the other hand $\pi_1((A, s), (A^g, t)) = (A^g, t)$ and $\pi_2((A, s), (A^g, t)) = (A, s)$, so that  the fibre
of 
$$
(\pi_1^*(\Det)^* \otimes \pi_2^*(\Det) \otimes \b^*(\cF/\cG) , \cA_0)
$$
at the point $((A, s), (A^g, t))$ is 
$$
 \Det_{(A^g, t)}^* \otimes \Det_{(A, s)}^{\vps}  \otimes \cF_{(A^g, s, t)}^{\vps}.
$$
Recall that $\hat g \in \widehat{\cG}$ is an automorphism $\Det \to \Det$ covering the action of $g \in \cG$ on $\cA_0$. Therefore, for every $s$ not in the spectrum of $D_{\! A}$, we have $\hat g(A) \colon \Det_{(A, s)} \to \Det_{(A^g, s)}$ and hence fiberwise we can define a map
$$
\varphi([(A, \hat g), z], (A, s,t) ) \colon \Det_{(A, s)} \to \Det_{(A^g, s)} 
$$
by $ \varphi{([(A, \hat g), z], (A, s,t))}  := \hat g (A) z^{-1}$. This is well-defined
under change of representative,
$$
([f\hat g, A, f(A)z], (A, s,t)) \mapsto  f(A)\hat g(A) f(A)^{-1}z ^{-1}= \hat g(A) z^{-1},
$$ 
so we have 
$$
 \varphi([(A, \hat g), z], (A, s,t)) \in \Det_{(A^g, s)}^{\vps} \otimes \Det_{(A, s)}^* 
\simeq\Det_{(A^g, t)}^{\vps} \otimes \Det_{(A, s)}^* \otimes \cF_{(A^g, s, t)}^* 
$$
using equation \eqref{eq:FMtriv}. We conclude that
$$
\varphi \colon \a^* (\ev_* (\cL)) \to \d (\Det) \otimes \b^*(\cF /\cG)^*
$$
as a map of spaces. However, note that the action of $z \in U(1)$ on $\a^* (\ev_* (\cL))$ gets mapped to the action of $z^{-1}$ on $\d (\Det) \otimes \b^*(\cF /\cG)^*$, so in fact
$$
\varphi \colon \a^*( \ev_* (\cL)) \to \d (\Det)^* \otimes \b^*(\cF /\cG)
$$
as a map of $U(1)$-bundles. That $\varphi$ is a bundle gerbe isomorphism follows from the fact that the bundle gerbe multiplication on $\ev_*(\cL)$ is given by the group multiplication in $\widehat \cG$, which is composition of maps $\Det \to \Det$, and that the isomorphism $\d (\Det) \simeq \cF$ is itself a bundle gerbe isomorphism.

Since $(\b^*(\cF/ \cG), \cA_0)$ differs from $(\delta (\Det)^* \otimes \b^*(\cF/\cG) , \cA_0)$ by a trivial bundle gerbe, they have the same Dixmier--Douady class,
$$
\DD(\a^*(\ev_*(\cL)), \cA_0) = \DD (\delta (\Det)^* \otimes \b^*(\cF/\cG) , \cA_0) =  \DD(\b^*(\cF/\cG) , \cA_0 ).
$$
By Proposition \ref{P:stable isos}, $\DD(\a^*(\ev_*(\cL)), \cA_0) = \DD(\ev_*(\cL), \cA)$ and $\DD(\b^*(\cF/\cG) , \cA_0 ) = \DD(\cF/\cG , \cA_0 /\cG)$, thus
$$
\DD(\ev_*(\cL), \cA) = \DD(\cF/\cG , \cA_0 /\cG).
$$
\end{proof}

\section{Relation to the caloron correspondence}\label{S:caloron gerbe}

We conclude this paper by providing another description of the lifting bundle gerbe $(\ev_* (\cL),\cA)$ using the caloron correspondence \cite{HekMurVoz}. 
\subsection{The caloron bundle gerbe}
The caloron correspondence is a natural bijection between isomorphism classes of $G$-bundles $\widetilde P$ on a product manifold $M\times X$ and $\cG$-bundles on $M$. Here $\cG$ is the gauge group of an auxiliary principal $G$-bundle $Q$ on the compact connected manifold $X$ and it is assumed that $\widetilde P$ is of type $Q$, meaning that the restriction $\widetilde P|_{m}$ is isomorphic to $Q$ for all $m\in M$. 
In more detail, starting with a $\cG$-bundle $P$ the associated $G$-bundle $\widetilde P$ on $M\times X$ is defined by
\[
\widetilde P = (P\times Q)/\cG. 
\]
Going in the other direction, applying the functor of $G$-equivariant maps $\Eq_G(Q, \cdot)$ to $\widetilde P$ one shows that the resulting bundle is indeed a $\cG$-bundle on $M$.

Noting that $\Map(\cA, U(1))$ is the gauge group of the trivial bundle $\cA\times U(1)$, we can apply the caloron transform to the abelian extension $\widehat \cG \to\cG$ which produces a circle bundle $\cC$ over the fibre product $\cA^{[2]} =\cA \times \cG$,
\[
\cC= (\widehat \cG  \times \cA\times U(1))/\Map(\cA, U(1)). 
\]
 More importantly, the group law in $\widehat \cG$ induces a bundle gerbe multiplication on $\cC$. To see this, regard $\widehat \cG$ as a \em left \rm principal $\Map(\cA, U(1))$-bundle and define the product
 \[
 [\hat g_1, A, z_1]\cdot [\hat g_2, A^{g_1}, z_2] = [\hat g_1\hat g_2, A, z_1z_2], 
 \]
 which on fibres takes the form
\[
\cC_{(A,g_1)}\otimes \cC_{(A^{g_1},g_2)} \cong \cC_{(A,g_1g_2)}.
 \]
Changing representatives we get
\[
[f_1\hat g_1, A, z_1 f_1(A)]\cdot [f_2 \hat g_2, A^{g_1}, z_2 f_2(A^{g_1})] = [f_1 \hat g_1 f_2 \hat g_1^{-1} \hat g_1 \hat g_2, A, z_1 z_2 f_1(A)f_2(A^{g_1})].
\]
For the bundle gerbe multiplication to be well-defined we need that
\[
(f_1 (\hat g_1 f_2 \hat g_1^{-1}))(A) = f_1(A) f_2(A^{g_1}), 
\]
but this holds since the adjoint action of $\widehat \cG$ induces the left $\cG$-module structure on $\Map(\cA,U(1))$. Note that for the multiplication to work, it is essential that $\Map(\cA, U(1))$ acts on the left on $\widehat \cG$. We call $(\cC,\cA)$ the \emph{caloron bundle gerbe}\footnote{Note that in a recent paper \cite{BMW} the authors also use the term caloron bundle gerbe but in a different context. That paper considers a $G$-bundle over $M \times S^1$ and uses the caloron correspondence to get an $LG$-bundle over $M$, where $LG$ is the loop group of $G$. The caloron bundle gerbe in their setting is essentially the lifting bundle gerbe for the central extension of $LG$.
} and it is not hard to see that $(\ev_*(\cL), \cA)$ is isomorphic to $(\cC, \cA)$. Indeed if $\hat g \in \widehat \cG$, $A \in \cA$ and $z \in U(1)$, then an element of $\ev_*(\cL)$ is given by $[(A, \hat g), z] = [(A, f\hat g), f(A)z]$ and an element of $\cC$ is a triple $[\hat g, A, z] = [f \hat g, A, f(A)z]$. The isomorphism is simply given  by
$$
\ev_*(\cL) = \cL\times_{\ev}  U(1) \to\cC, \ \ \ [(A,\hat g),z]\mapsto [\hat g, A, z].
$$

\subsection{Conclusion}

The point of view explained above lends itself naturally to the study of the geometry of the FMS bundle gerbe in the following way. The strength of the caloron correspondence is that it holds not only at the level of bundles but also at the level of connections. More precisely, if $\widetilde P \to M \times X$ is a $G$-bundle of type $Q$ with connection, then the corresponding $\cG$-bundle $P \to M$ inherits a connection and also the additional data of a \emph{Higgs field}. The Higgs field is a section of $Q \times_\cG \cA$, where $\cA$ is the space of connections on $Q$. The result from \cite{HekMurVoz} implies that a pair given by a connection and a Higgs field determines the complete geometry of $P$ in the sense that there is an equivalence of categories between $G$-bundles of type $Q$ with connection and $\cG$-bundles with connection and Higgs field. In particular, this means that reciprocally one can   construct a connection on the $G$-bundle given a connection and Higgs field on the $\cG$-bundle.

An interesting open problem is to use this geometric caloron correspondence to give a bundle gerbe connection and curving on the caloron bundle gerbe (and hence the lifting bundle gerbe) using a connection and Higgs field on the abelian extension $\widehat \cG \to \cG$. In particular, the group $\widehat \cG$ carries a natural connection $\alpha$ defined by \cite{Hek10} 
$$
\alpha = \Ad_{\hat{g}^{-1}} \pr (\theta),
 $$
where $\theta$ is the right Maurer-Cartan form on $\widehat \cG$ and $\pr$ denotes the projection onto the Lie algebra of $\Map(\cA, U(1))$.

\end{document}